\documentclass[UTF-8,reqno]{amsart}
\usepackage{enumerate}
\usepackage{verbatim,listings}
\usepackage{amssymb,url,color, booktabs}

\usepackage{mathrsfs}
\usepackage{amsmath}
\usepackage{verbatim}
\usepackage{fancyhdr}
\usepackage[nobysame]{amsrefs}
\BibSpec{article}{%
+{}{\PrintAuthors} {author}
+{,}{ \textrm} {title}
+{.}{ \textit} {journal}
+{,}{ \textbf} {volume}
+{}{ \parenthesize} {date}
+{,}{ } {pages}
+{.}{ arXiv:} {eprint}
+{.}{} {transition}
}
\BibSpec{book}{%
+{}{\PrintAuthors} {author}
+{,}{ \textit} {title}
+{.}{ \textrm} {series} 
+{,}{ Vol.} {volume}
+{.}{ } {publisher}
+{,}{ } {date}
+{.}{} {transition}
}
\usepackage{color}
\usepackage{tikz}
\usepackage{graphicx}
\usepackage[colorlinks=true]{hyperref}
\hypersetup{
    linkcolor=blue,          
    citecolor=red,        
    filecolor=blue,      
    urlcolor=cyan
}

\newcommand{\red}{\color{red}}

\numberwithin{equation}{section}

\newtheorem{theorem}{Theorem}[section]

\newtheorem{lemma}[theorem]{Lemma}
\newtheorem{definition}[theorem]{Definition}
\newtheorem{proposition}[theorem]{Proposition}

\newenvironment{proof of theorem 1.2 and 1.3}{{\it Proof of Theorem 1.2 and 1.3}.}{{\hfill 	
$\square$\hskip - \parfillskip}}
\newenvironment{proof of theorem 1.4}{{\it Proof of Theorem 1.4}.}{{\hfill 	
		$\square$\hskip - \parfillskip}}
\newenvironment{proof of theorem 1.5}{{\it Proof of Theorem 1.5}.}{{\hfill 	
		$\square$\hskip - \parfillskip}}
\newenvironment{proof of theorem 1.6}{{\it Proof of Theorem 1.6}.}{{\hfill 	
		$\square$\hskip - \parfillskip}}

\makeatletter

\newcommand{\Rmnum}[1]{\expandafter\@slowromancap\romannumeral #1@}
\makeatother

\usepackage{caption}
\def\geq{\geqslant}
\def\leq{\leqslant}
\def\ge{\geqslant}
\def\le{\leqslant}

\def\R{\mathbb{R}}
\def\Sn{\mathbb{S}^{n}}
\def\={&\!\!=\!\!&}
\def\al{\alpha}
\def\be{\beta}
\def\ga{\gamma}
\def\pa{\partial}
\def\lbd{\lambda}

\def\ep{\epsilon}

\def\Ctheta{\mathcal{C_{\theta}}}

\allowdisplaybreaks

	\tikzset{
	pattern size/.store in=\mcSize,
	pattern size = 5pt,
	pattern thickness/.store in=\mcThickness,
	pattern thickness = 0.3pt,
	pattern radius/.store in=\mcRadius,
	pattern radius = 1pt}
\makeatletter
\pgfutil@ifundefined{pgf@pattern@name@_ak5ydppfi}{
	\pgfdeclarepatternformonly[\mcThickness,\mcSize]{_ak5ydppfi}
	{\pgfqpoint{0pt}{0pt}}
	{\pgfpoint{\mcSize+\mcThickness}{\mcSize+\mcThickness}}
	{\pgfpoint{\mcSize}{\mcSize}}
	{
		\pgfsetcolor{\tikz@pattern@color}
		\pgfsetlinewidth{\mcThickness}
		\pgfpathmoveto{\pgfqpoint{0pt}{0pt}}
		\pgfpathlineto{\pgfpoint{\mcSize+\mcThickness}{\mcSize+\mcThickness}}
		\pgfusepath{stroke}
}}
\makeatother
\tikzset{every picture/.style={line width=0.75pt}} 

\begin{document}
\title{The capillary $L_{p}$ dual Christoffel-Minkowski problem for $1<p<q\leq k+1$ with $1\leq k\leq n$}\thanks{\it {The research is partially supported by NSFC (No. 12261105).}}

\author{Guanghan Li}
\address{School of Mathematics and Statistics\\                	
	Wuhan University\\                Wuhan 430072, P. R. China}
\email{ghli@whu.edu.cn}
\author{Mengliang Liu$^{*}$}\thanks{* Corresponding author}
\address{School of Mathematics and Statistics\\                	
	Wuhan University\\                Wuhan 430072, P. R. China}
\email{liumengliang@whu.edu.cn}

\begin{abstract}
This paper is concerned with a capillary $L_p$ Christoffel--Minkowski-type problem involving the prescription of a class of $L_p$ geometric measures formed by the $k$-th capillary area measure and the $q$-th dual curvature measure. By establishing a gradient estimate, we prove the existence of an even, smooth, strictly convex solution in the parameter range $1<p<q\leq k+1$, where $1\leq k\leq n$.
\end{abstract}

\maketitle
\setcounter{tocdepth}{2}

\section{Introduction}
The Brunn--Minkowski theory connects convex geometry, geometric analysis, and fully nonlinear elliptic equations. At its core are the surface area measures and their higher-order analogues, which encode curvature via the Gauss map. The inverse problem of prescribing these measures is the Christoffel--Minkowski problem. In the smooth strictly convex setting, with support function $h$, the problem reduces to
\begin{equation}\label{eq:classical-CM}
\sigma_k(\nabla^2 h + h g_{\Sn})=\varphi \quad\text{on }\Sn,
\end{equation}
$1\leq k\leq n$, with $\sigma_k$ the $k$-th elementary symmetric function. Cases $k=1$ and $k=n$ are the Christoffel and Minkowski problems, classical since \cite{Christoffel1865,Minkowski1897,Aleksandrov1956,Firey1967,Firey1968,Nirenberg1953,Pogorelov1978,ChengYau1976,Schneider2014}. For $1<k<n$, Guan--Ma \cite{GuanMa2003} showed that spherical convexity of $\varphi^{-1/k}$ implies strict convexity of solutions, with further results in \cite{GuanLinMa2006,GuanMaZhou2006}.

The $L_p$-Brunn--Minkowski theory by Lutwak \cite{Lutwak1993} replaces surface area measures by $L_p$ counterparts, yielding
\begin{equation}\label{eq:Lp-CM}
\sigma_k(\nabla^2 h+h g_{\Sn})=h^{p-1}\varphi \quad\text{on }\Sn,
\end{equation}
the $L_p$-Christoffel--Minkowski problem. For $k=n$, this is the $L_p$-Minkowski problem, while intermediate cases were studied in \cite{HuMaShen2004} for $p>k+1$ and \cite{GuanXia2018} for $1<p<k+1$.

The dual Brunn--Minkowski theory by Huang--Lutwak--Yang--Zhang \cite{HuangLutwakYangZhang2016} introduced dual curvature measures. Lutwak--Yang--Zhang \cite{LutwakYangZhang2018} unified both theories via $L_p$-dual curvature measures, leading to
\begin{equation}\label{eq:closed-Lp-dual-CM}
\sigma_k(\nabla^2 h+h g_{\Sn})=f h^{p-1}(h^2+|\nabla h|^2)^{\frac{k+1-q}{2}} \quad\text{on }\Sn,
\end{equation}
with $L_p$-parameter $p$ and dual parameter $q$. Special cases include \eqref{eq:classical-CM} when $(p,q)=(1,k+1)$, \eqref{eq:Lp-CM} when $q=k+1$, and the $L_p$-dual Minkowski problem when $k=n$.

We recall known results for the $L_p$-dual Christoffel--Minkowski problem.
For the general range $1\leq k\leq n$, Ding and Li \cite{DingLi2023} studied inverse curvature flows for \eqref{eq:closed-Lp-dual-CM}.
Chen--Tu--Xiang \cite{ChenTuXiang2025} treated the case $p\ge q$ with $p\ge1$. They proved existence and uniqueness of positive strictly spherical convex solutions for $p>q$, while for $p=q>1$ uniqueness holds up to dilation.
Cabezas-Moreno and Hu \cite{CabezasMorenoHu2025} proved the existence of even, smooth, strictly convex solutions for $1<p<q\le k+1$ via gradient estimates and degree theory.
Thus for $p>1$, known existence results cover $p\ge q$ and the subrange $1<p<q\le k+1$.
For the top-order case $k=n$, the equation becomes the smooth $L_p$-dual Minkowski problem, which has been extensively studied in \cite{HuangZhao2018, BoroczkyFodor2019, ChenHuangZhao2019, ChenLi2021, HuangLutwakYangZhang2016, BoroczkyLutwakYangZhangZhao2017, Zhao2017, LiLiuLu2019}. Uniqueness for $p>q$ and non-uniqueness for $p<0<q$ are established in \cite{HuangZhao2018} and \cite{LiLiuLu2019} respectively.

The classical Christoffel--Minkowski problem and its $L_p$ extension lead naturally to analogous questions for capillary hypersurfaces, as discussed in Section \ref{sec2.1}. The capillary Minkowski problem and its $L_p$ variants have developed rapidly in recent years. Mei, Wang, and Weng introduced and solved the capillary Minkowski problem in the smooth setting \cite{MeiWangWeng2025Minkowski}. They also proposed the capillary $L_p$-Minkowski problem for $p\ge1$. This problem is a natural Robin-boundary analogue of Lutwak's $L_p$-Minkowski problem \cite{MeiWangWeng2025LpMinkowski}. 
The capillary Christoffel--Minkowski problem was studied by Hu, Ivaki, and Scheuer \cite{HuIvakiScheuer2025}. They established an analogue of the Guan-Ma \cite{GuanMa2003} theorem in the half-space. Under a capillary spherical convexity condition on $\varphi^{-1/k}$, a prescribed $\sigma_k$-curvature is realized by a strictly convex capillary hypersurface. This work was also carried out independently around the same time by Mei, Wang, and Weng \cite{mwwcvpde}. 
The capillary $L_p$-Christoffel--Minkowski problem was then solved by Hu and Ivaki \cite{HuIvaki2025LpCM} in the range $1<p<k+1$ in the class of even hypersurfaces. Their proof combines non-collapsing estimates, curvature estimates and a capillary constant rank theorem. 
Closely related techniques appear in the capillary $L_p$-curvature problem \cite{HuIvaki2026LpCurvature}. For general $1\leq k < n$, the capillary $L_p$ dual Christoffel--Minkowski problem was considered by Li and Liu \cite{liliu} in the complementary range $p\geq q$.

Motivated by these developments, we consider the following capillary $L_p$-dual Christoffel--Minkowski problem. Given $1\le k\le n$, parameters $p,q\in\R$, a contact angle $\theta\in(0,\pi/2)$, and a positive smooth function $f$ on $\Ctheta$, find a positive strictly convex capillary support function $h$ satisfying
\begin{equation}\label{eq:capillary-Lp-dual-CM}
		\renewcommand{\arraystretch}{1.5}
		\left\{
		\begin{array}{rll}
			\sigma_k\big(\nabla^{2}h+hI\big)&=fh^{p-1}\big(h^2+|\nabla h|^2\big)^{\frac{k+1-q}{2}},& \text{in }\Ctheta,\\
		\nabla_{\mu}h&=\cot\theta h,& \text{on }\partial\Ctheta,
			\end{array}	
			\right.
	\end{equation}
where $\mathcal{C}_\theta$ is the capillary spherical cap, see Section \ref{sec 2}. This is a natural Robin boundary version of the classical $L_p$-dual Christoffel--Minkowski problem.
Here all derivatives and norms are taken with respect to the standard metric $g$ on $\Ctheta$. This equation simultaneously extends the closed $L_p$-dual Christoffel--Minkowski equation \eqref{eq:closed-Lp-dual-CM} and the capillary $L_p$-Christoffel--Minkowski equation. Indeed, when $q=k+1$, the gradient factor disappears and \eqref{eq:capillary-Lp-dual-CM} becomes the capillary $L_p$-Christoffel--Minkowski equation studied in \cite{HuIvaki2025LpCM}. When $k=n$, it becomes the capillary $L_p$-dual Minkowski equation, namely
\begin{equation}\label{eq:capillary-Lp-dual-Minkowski}
		\renewcommand{\arraystretch}{1.5}
		\left\{
		\begin{array}{rll}
			\det\big(\nabla^2h+hI\big)&=fh^{p-1}\big(h^2+|\nabla h|^2\big)^{\frac{n+1-q}{2}},& \text{in }\Ctheta,\\
		\nabla_{\mu}h&=\cot\theta h,& \text{on }\partial\Ctheta.
			\end{array}	
			\right.
	\end{equation}

Thus, the case $k=n$ is a capillary Monge-Amp\`{e}re equation with Robin boundary condition. Gao \cite{Gao2025CapillaryDual} studied \eqref{eq:capillary-Lp-dual-Minkowski} in the range $p>q$, $q\le1$, and proved existence and uniqueness for smooth strictly convex capillary solutions. Hu and Yang \cite{HuYang2026John} later used a capillary John ellipsoid theorem, a non-collapsing estimate, a gradient estimate and a refined $C^2$-estimate to obtain results in the three-dimensional half-space $\R^3_+$. They proved existence for $1<p\le q\le3$ under evenness assumptions in the appropriate cases and improved the $p>q$ result by removing the restriction $q\le1$ in that dimension. In a related direction, Hu, Hu, and Ivaki \cite{HuHuIvaki2025Flow} developed anisotropic capillary Gauss curvature flows and applied them to the capillary $L_p$-Minkowski problem, which corresponds to the special case $q=n+1$ of \eqref{eq:capillary-Lp-dual-Minkowski}.

The passage from the closed equation \eqref{eq:closed-Lp-dual-CM} to the capillary boundary value problem \eqref{eq:capillary-Lp-dual-CM} introduces several difficulties.
First, the dual factor $(h^2+|\nabla h|^2)^{(k+1-q)/2}$ couples the support function and its gradient. This coupling interacts with the boundary terms arising from $\nabla_{\mu}h=\cot\theta h$. For the $C^{2}$-estimate, the auxiliary function used in \cite{CabezasMorenoHu2025} does not apply directly. We must introduce a barrier function to ensure that the maximum of the auxiliary function cannot occur on the boundary, while preserving the maximum principle in the interior.
Second, the main a priori estimates for \eqref{eq:capillary-Lp-dual-CM} require a capillary version of the compact closed-sphere arguments. In particular, lower bounds for $h$ and non-collapsing estimates are delicate in the half-space.
Recent tools address these issues, including weighted capillary gradient estimates and the capillary John ellipsoid theorem \cite{HuIvaki2026LpCurvature,HuYang2026John}.
The geometric condition on the prescribed function must also reflect the boundary geometry. In the closed Christoffel--Minkowski theory, spherical convexity conditions such as
$$
\nabla^2\varphi^{-1/k}+g_{\Sn}\varphi^{-1/k}\ge0
$$
are natural and are closely related to the constant rank method. In the $L_p$-Christoffel--Minkowski problem, the corresponding exponent becomes $(p+k-1)^{-1}$. In the capillary setting, this interior convexity condition is supplemented by a boundary compatibility condition of Robin type, for example
$$
\nabla^2 f^{-\frac{1}{p+k-1}}
+g f^{-\frac{1}{p+k-1}}\ge0
\quad\text{in }\Ctheta,
\qquad
\nabla_{\mu}f^{-\frac{1}{p+k-1}}
\le\cot\theta f^{-\frac{1}{p+k-1}}
\quad\text{on }\partial\Ctheta .
$$
Such assumptions appear naturally in the capillary Christoffel--Minkowski and capillary $L_p$-Christoffel--Minkowski problems \cite{HuIvakiScheuer2025,HuIvaki2025LpCM,mwwcvpde}. 
\begin{theorem}\label{thm1}
Let $1\leq k <n$, $1<p<q\leq k+1$ and $\theta \in (0,\frac{\pi}{2})$. Assume $f \in C^{\infty}(\Ctheta)$ is a positive function satisfying 
$$f(-\xi_{1},\cdots,-\xi_{n},\xi_{n+1})=f(\xi_{1},\cdots,\xi_{n},\xi_{n+1})\quad \forall \xi \in \Ctheta.$$
$$
\nabla^{2}{f}^{-\frac{1}{p+k-1}}+gf^{-\frac{1}{p+k-1}}\geq 0
\quad\rm{in }\  \Ctheta,
$$
and the boundary condition 
$$\ \ \ \nabla_{\mu}f^{-\frac{1}{p+k-1}}\leq \cot\theta f^{-\frac{1}{p+k-1}}\quad\rm{on }\  \pa\Ctheta.$$
Then there exists a smooth, even, strictly convex, capillary hypersurface $\Sigma\subset\overline{\R^{n+1}_{+}}$ with contact angle $\theta$ whose capillary support function $h$ solves Eq. \eqref{eq:capillary-Lp-dual-CM}.
\end{theorem}

The present paper is organized as follows. In Section \ref{sec 2} we recall basic facts on capillary convex hypersurfaces, the capillary Gauss map, capillary support functions, and elementary symmetric functions. Section \ref{sec 3} is devoted to a priori estimates for solutions of \eqref{eq:capillary-Lp-dual-CM}, including gradient and non-collapsing estimates adapted to the Robin boundary condition. We also establish higher-order estimates and a constant rank theorem ensuring strict convexity. The proof of the main existence result is completed in Section \ref{sec 4} via a degree-theoretic argument.
Throughout the paper, given an orthonormal frame $\{e_i\}_{i=1}^n$ on $\Ctheta$, we write $h_{ij}$ for $\nabla^2 h(e_i,e_j)$ and $b_{ij,k}$ for $\nabla_{e_k}b_{ij}$, with analogous notation for higher derivatives. We use the Einstein summation convention: repeated indices are summed over regardless of position. When ambiguity may arise, summation is indicated explicitly.

\section{Preliminaries}\label{sec 2}
In this section, we collect some basic facts about capillary convex bodies and $k$-th elementary symmetric functions.
\subsection{Basics of capillary geometry}\label{sec2.1}
Let $\{E_i\}_{i=1}^{n+1}$ be the standard orthonormal basis of $\R^{n+1}$, $\R^{n+1}_{+}=\{x\in\R^{n+1}| x\cdot E_{n+1}>0\}$ be the upper Euclidean half-space. Let $\Sigma\subset\overline{\R^{n+1}_{+}}$ be a properly embedded, smooth compact hypersurface with boundary such that
	$$int(\Sigma)\subset\R^{n+1}_{+}\qquad \mathrm{and} \qquad\partial\Sigma\subset\partial\R^{n+1}_{+}.$$
	We call $\Sigma\subset\overline{\R^{n+1}_{+}}$ a {\emph{capillary hypersurface}} if $\Sigma$ intersects $\partial\R^{n+1}_{+}$ at a constant contact angle $\theta\in (0,\pi)$. Let $\nu$ be the unit outward normal (the Gauss map) of $\Sigma$ with respect to the domain $\widehat{\Sigma}$. Here $\widehat{\Sigma}$ is the bounded closed region in $\overline{\R^{n+1}_{+}}$ enclosed by $\Sigma$ and $\partial\R^{n+1}_{+}$. The contact angle $\theta$ is defined by 
	$$\cos(\pi-\theta) = \langle \nu , e\rangle,\ \mathrm{ on }\ \Sigma\cap \partial\R^{n+1}_{+},$$
	where $e:= -E_{n+1}$ is the unit outward normal of $\partial\R^{n+1}_{+}$.  Let $\mu$ be the unit outward co-normal of $\partial\Sigma$ in $\Sigma$ and let $\bar{\nu}$ be the unit normal of $\partial\Sigma\subset\overline{\mathbb{R}^{n+1}_{+}}$. Assume that \{$\nu,\mu$\} and \{$\bar{\nu},e$\} have the same orientation in the normal bundle of $\partial\Sigma\subset\overline{\mathbb{R}^{n+1}_{+}}$. Then the following relation holds:
\begin{equation}\label{normal transform}
\left\{\begin{aligned}
    \nu=&-\cos\theta\, e+\sin\theta\, \bar{\nu},\\
	\mu=&\sin\theta \,e+\cos\theta \,\bar{\nu}.
	\end{aligned}\right.
\end{equation}

	If $\Sigma$ is convex, then the Gauss image $\nu(\Sigma)$ of $\Sigma$ lies in the spherical cap
	$$\Sn_{\theta}:=\left\{x\in \Sn | \langle x, E_{n+1}\rangle \geq \cos\theta\right\}.$$
	Instead of the usual Gauss map $\nu$, it is more convenient to use the following map
	$$\tilde{\nu}:= T \circ \nu: \Sigma\to \Ctheta,$$
	where $\Ctheta$ is the spherical cap defined by
	$$\Ctheta:= \left\{\xi\in\overline{\R^{n+1}_{+}} |~ |\xi-\cos\theta\cdot e|=1\right\}.$$
Here, $T: \Sn_{\theta}\to \Ctheta$ denotes the vertical translation defined by $T(z)=z+\cos\theta\cdot e$. The diffeomorphism $\tilde{\nu}$ is called the {\emph{capillary Gauss map}} of $\Sigma$. Thus we can reparameterize $\Sigma$ using its inverse on $\Ctheta$ (see \cite{Xia-arxiv}*{Lemma 2.2}). 

Taking $X=\tilde{\nu}^{-1}$, the capillary support function of a strictly convex capillary hypersurface $\Sigma$, denoted by $h = h_{\Sigma} : \mathcal{C}_\theta \to \mathbb{R}$, is defined by
\begin{equation}
 \label{s2:def-capillary-spt}
h(\xi) := \langle X(\xi) , \nu\big(X(\xi)\big)\rangle = \langle\tilde{\nu}^{-1}(\xi)  ,\xi - \cos \theta\, e\rangle, \ \forall \, \xi \in \mathcal{C}^n_\theta.
\end{equation}
When $\Sigma=\mathcal{C}^n_{\theta}$, we have $X(\xi)=\xi$ for all $\xi\in \mathcal{C}^{n}_\theta$. It follows that
\begin{equation}
 \ell(\xi):=\sin^2\theta+\cos \theta \langle \xi, e \rangle.
\end{equation}
A direct calculation shows that $\nabla_{ij}^{2}\ell+\ell\delta_{ij}=\delta_{ij}$ and $1-\cos\theta\leq \ell\leq \sin^{2}\theta$.
Instead of the usual support function $h$, it is convenient to introduce the {\emph{capillary support}} 
$$u(\xi):=\frac{h(\xi)}{\ell(\xi)}.$$
The following properties hold for the support function $h$, the capillary support function $u$, and the spherical cap support function $\ell$. For the proof, we refer to \cite{MeiWangWeng2025Minkowski}.
\begin{lemma}\label{lem1}
Along the boundary $\pa\Ctheta$, we choose an orthonormal frame $\{e_{i}\}_{i=1}^{n}$ with $e_{n}=\mu$, where $\mu$ is the unit outer normal of $\pa\Ctheta\subset \Ctheta$. The following boundary conditions hold on $\partial \mathcal{C}_\theta$:
\begin{enumerate}
\item $\nabla_{\mu}h=\cot\theta h$,\\
\item $h_{\al n}=0,\quad\al=1,2,\cdots,n-1$,\\
\item$ \nabla_{\mu}u=0$,\ $\nabla_{\mu}\ell=\cot\theta \ell$,\\
\item $u_{\al n}=-\cot\theta u_{\al},\quad\al=1,2,\cdots,n-1$.
\end{enumerate}
\end{lemma}
\subsection{$k$-th elementary symmetric function}
Let $Z=\{z_{ij}\}$ be an $n\times n$ symmetric matrix,
$$
\sigma_{k}(Z)=\sigma_{k}(\Lambda(Z))=\sum_{1\leq i_{1}<i_{2}\ldots <i_{k}\leq n}\Lambda_{i_{1}}\Lambda_{i_{2}}\ldots \Lambda_{i_{k}},
$$
where $\Lambda:=\Lambda(Z)=(\Lambda_{1},\ldots,\Lambda_{n})\in \mathbb{R}^{n}$ is the set of eigenvalues of $Z$.
\begin{definition}
Let $1\leq k\leq n$ and $\Gamma_{k}$ be a cone in $\mathbb{R}^{n}$ defined as
$$
\Gamma_{k}=\{\Lambda \in \mathbb{R}^{n}:\sigma_{i}(\Lambda)>0, \ \forall 1\leq i \leq k\}.
$$
\end{definition}

\begin{definition}
A function $h\in C^{2}(\Ctheta)$ of \eqref{eq:capillary-Lp-dual-CM} is called a $k$-admissible solution if
$$
b=\nabla^{2}h(\xi)+h(\xi)I\in \Gamma_{k}
$$
for all points $\xi\in \Ctheta$. If $b\in \Gamma_{n}$, then $h$ is strictly (spherical) convex.
\end{definition}
We denote by $\sigma_{k}(\Lambda|i)$ the symmetric function with $\Lambda_{i}=0$. We list below some basic properties of the $k$-th elementary symmetric functions that will be used throughout the paper.
\begin{proposition}\label{basic}
Let $\Lambda=(\Lambda_{1},\cdots, \Lambda_{n})\in \mathbb{R}^{n}$ and $k=0,1,\cdots, n$. Then,
\begin{enumerate}
\item[(i)] $\sigma_{k+1}(\Lambda)=\sigma_{k+1}(\Lambda|i)+\lambda_{i}\sigma_{k}(\Lambda|i), \quad \forall 1\leq i\leq n$.
\item[(ii)] $\sum_{i=1}^{n}\Lambda_{i}\sigma_{k}(\Lambda|i)=(k+1)\sigma_{k+1}(\Lambda)$.

\item[(iii)]$\sum_{i=1}^{n}\sigma_{k}(\Lambda|i)=(n-k)\sigma_{k}(\Lambda)$.

\item[(iv)] $\sum^{n}_{i=1}\Lambda^{2}_{i}\sigma_{k-1}(\Lambda|i)=\sigma_{1}(\Lambda)\sigma_{k}(\Lambda)-(k+1)\sigma_{k+1}(\Lambda)$.

\end{enumerate}
\end{proposition}

\begin{proposition}[Newton-Maclaurin inequality]\label{NM} For $\Lambda\in \Gamma_{k}$ and $1\leq l \leq k\leq n$, we have
$$
\left[\frac{\sigma_{k}(\Lambda)}{\binom{n}{k}}\right]^{\frac{1}{k}}\leq \left[\frac{\sigma_{l}(\Lambda)}{\binom{n}{l}}\right]^{\frac{1}{l}}.$$
\end{proposition}
The following proposition is well known, see e.g., \cite{Gcpam}*{Lemma 2.1}.
\begin{proposition}\label{g}
Denote by $\rm{Sym(n)}$ the set of all $n\times n$ symmetric matrices. Let $F$ be a $C^{2}$ symmetric function defined in some open subset $\Psi\subset \rm{Sym(n)}$. For a diagonal matrix $Z \in \Psi$ with distinct eigenvalues, write $F(Z)=f(\kappa(Z))$, where $\kappa(Z)=(\kappa_{1},\cdots,\kappa_{n})$ are the eigenvalues of $Z$. Let $\ddot{F}(Y,Y)$ be the second derivative of a $C^{2}$ symmetric $F$ in direction $Y \in\rm{Sym(n)}$; then
$$\ddot{F}(Y,Y)=\sum_{j,k=1}^{n}\ddot{f}^{jk}Y_{jj}Y_{kk}+2\sum_{j<k}\frac{\dot{f}^{j}-\dot{f}^{k}}{\lambda_{j}-\lambda_{k}}Y^{2}_{jk},$$
where the first derivative of $F$ satisfies $\dot{F}^{ij}(Z)=\dot{f}^{i}\delta_{ij}$.
\end{proposition}

\section{A priori estimates}\label{sec 3}
For the capillary $L_p$ dual Christoffel--Minkowski problem, the key step is to obtain a priori estimates, particularly the lower bound of the support function and the $C^2$ estimate. As in the classical closed case, when $p<q$ we first obtain a non-collapsing estimate via gradient estimates combined with the capillary John ellipsoid theorem, and then derive a uniform $C^0$ estimate by the maximum principle. Let $R:=\underset{\Ctheta}{\max}\ h$, $r:=\underset{\Ctheta}{\min}\ h$ and $\rho^{2}:=h^{2}+|\nabla h|^{2}$.

\begin{lemma}\label{lem2}
Let $\theta \in (0,\pi/2)$, $1\leq k\leq n$, $1<p\leq k+1$ and $q\leq k+1$. Suppose $h$ is a positive, smooth and strictly convex solution to Eq. \eqref{eq:capillary-Lp-dual-CM}. For any $0<\ga<\min\{2(p-1)/n,2\}$, there exists a positive constant $\hat{C}$ depending on $\ga,n,p,q,\theta, ||f||_{C^{1}(\Ctheta)}$ such that 
\begin{equation}\label{gradient estimate}
\frac{|\nabla h|^{2}}{h^{\ga}}\leq \hat{C}R^{2-\ga}.\ 
\end{equation}
\end{lemma}
\begin{proof}
We define the auxiliary function $$\Psi:=\frac{\ell^{2}|\nabla u|^{2}}{h^{\ga}}=\frac{\ell^{\be}|\nabla u|^{2}}{u^{\ga}},$$
where $\be=2-\ga>0$.
To prove \eqref{gradient estimate}, it suffices to show that
\begin{equation}\label{1}
\Psi \leq \hat{C} \left( \max_{\Ctheta} \ell \right)^{\beta} \left( \max_{\Ctheta} u \right)^{\beta},
\end{equation}
since $h = u \ell$ and $|\nabla \ell| \leq C(\theta)$.
Suppose \eqref{1} is not true and $\Psi$ attains its maximum at $\xi_{0} \in \Ctheta$, then at $\xi_{0}$, we have 
\begin{equation}\label{assume}
\Psi(\xi_{0})>\hat{C} \left( \max_{\Ctheta} \ell \right)^{\beta} \left( \max_{\Ctheta} u \right)^{\beta} \Leftrightarrow\ \frac{|\nabla u|^{2}}{u^{2}}(\xi_{0})>\hat{C}.
\end{equation}

{\bf Case I:} $\xi_0\in \partial \mathcal{C}_{\theta}$. Let $\{e_i\}^{n}_{i=1}$ be an orthonormal frame around $\xi_0$ such that $e_n=\mu$. The maximum condition and Lemma \ref{lem1} imply 
\begin{equation}
   0\leq \nabla_{\mu}\log \left(\frac{\ell^{\beta}|\nabla u|^2}{u^{\gamma}} \right) =\frac{2u_\al u_{\al n}}{|\nabla u|^2}+\beta \frac{\nabla_{\mu}\ell}{\ell}-\gamma\frac{\nabla_{\mu}u}{u} =-\gamma \cot \theta <0 .\notag
\end{equation}
This yields a contradiction. Hence $\xi_{0}\in \Ctheta\setminus\pa\Ctheta$.

{\bf Case II:} $\xi_{0}\in \Ctheta\setminus\pa\Ctheta$. The maximum condition at $\xi_{0}$ implies
\begin{equation}\label{2}
0=\nabla_i \log \Phi\Leftrightarrow\  \frac{2u_{m}u_{mi}}{|\nabla u|^{2}}=\ga \frac{u_{i}}{u}-\be\frac{\ell_{i}}{\ell},
\end{equation}
and
\begin{align*}
0\geq \nabla^2_{ij}(\log \Phi)&=\frac{2\sum_{m}(u_{m}u_{mij}+u_{mi}u_{mj})}{|\nabla u|^2} -(\beta+\beta^2)\frac{\ell_i \ell_j}{\ell^2}+\beta\gamma\frac{u_i \ell_j+u_j\ell_i}{u\ell}\\
&+(\gamma-\gamma^2)\frac{u_i u_j}{u^2}+ \beta \frac{1-\ell}{\ell}\delta_{ij}-\gamma \frac{u_{ij}}{u},
\end{align*}
where we used $\ell_{ij}+\ell \delta_{ij}=\delta_{ij}$. Recall that
$$\sigma_{k}^{ij}b_{ij}=\sigma_{k}^{ij}(\ell u_{ij}+u_{i}\ell_{j}+u_{j}\ell_{i}+u\delta_{ij})=kfh^{p-1}\rho^{k+1-q}$$
and hence
\begin{equation}\label{9}
\sigma_{k}^{ij}u_{ij}=\frac{1}{\ell}\Big(fh^{p-1}\rho^{k+1-q}-2\sigma_{k}^{ij}u_{i}\ell_{j}-u\sigma_{k}^{ij}\delta_{ij}\Big).
\end{equation}
Then we have that
\begin{align}
0&\geq \frac{2}{|\nabla u|^2} \left(\sigma^{ij}_{k}u_mu_{mij}+\sigma^{ij}_{k}u_{mi}u_{mj}\right)-(\beta +\beta^2)\sigma^{ij}_{k}\frac{\ell_i\ell_j}{\ell^2}\notag\\
&\quad+(2\beta \gamma +2 \gamma)\sigma^{ij}_{k}\frac{u_i\ell_j}{u\ell}+(\gamma-\gamma^2)\sigma^{ij}_{k}\frac{u_iu_j}{u^2}+\left( \beta \frac{1-\ell}{\ell}+\frac{\gamma}{\ell}\right)\sigma^{ij}_{k}\delta_{ij}\notag\\
&\quad-k\gamma fh^{p-2}\rho^{k+1-q}.\label{4}
\end{align}
We now focus on the first term on the right-hand side of \eqref{4}, namely $\frac{2}{|\nabla u|^{2}}\sigma_{k}^{ij}u_{m}u_{mij}$.
Note that
\begin{align}\label{5}
\sigma_{k}^{ij}b_{ij,m}&=\sigma_{k}^{ij}\Big(\ell_{m}u_{ij}+\ell u_{ijm}+2\ell_{im}u_{j}+2u_{im}\ell_{j}+u_{m}\delta_{ij}\Big)\notag\\
&=f_{m}h^{p-1}\rho^{k+1-q}+(p-1)fh^{p-2}h_{m}\rho^{k+1-q}\notag\\
&+(k+1-q)\rho^{k-q}fh^{p-1}\frac{hh_{m}+h_{r}h_{rm}}{\rho}.
\end{align}
Multiplying \eqref{5} by $\frac{u_{m}}{\ell}$ and summing over the index $m$, we get that
\begin{equation*}
\frac{\sigma_{k}^{ij}}{\ell}\Big(\ell_{m}u_{m}u_{ij}+\ell u_{ijm}u_{m}+2\ell_{im}u_{m}u_{j}+2u_{im}u_{m}u_{j}+|\nabla u|^{2}\Big)=\frac{\sigma_{k}^{ij}}{\ell}b_{ij,m}u_{m}.
\end{equation*}
Multiplying \eqref{2} by $\sigma_{k}^{ij}\frac{\ell_{j}}{\ell}$ and summing over the index $j$, we also get that
\begin{equation*}
-\frac{2}{\ell}\sigma_{k}^{ij}\ell_{j}u_{m}u_{mi}=\frac{|\nabla u|^{2}}{\ell^{2}}\be \sigma_{k}^{ij}\ell_{i}\ell_{j}-\frac{|\nabla u|^{2}}{u\ell}\ga\sigma_{k}^{ij}u_{i}\ell_{j}.
\end{equation*}
Through a similar procedure to \eqref{9}, we can obtain
\begin{equation*}
-\frac{1}{\ell}\sigma_{k}^{ij}u_{ij}u_{m}\ell_{m}=\frac{1}{\ell^{2}}(2\sigma_{k}^{ij}u_{i}\ell_{j}+u\sigma^{ij}_{k}\delta_{ij})u_{m}\ell_{m}-\frac{k}{\ell^{2}}fh^{p-1}\rho^{k+1-q}u_{m}\ell_{m}.
\end{equation*}
Since $h = u \ell$, a direct calculation yields
\begin{equation*}
hu_{m}h_{m}=\ell u_{m}u_{r}b_{rm}+uu_{m}\ell_{r}b_{rm}.
\end{equation*}
Combining \eqref{5} with the equality obtained from the above calculation, we have 
\begin{align}
\sigma_{k}^{ij}u_{m}u_{ijm}&=\frac{f_{m}u_{m}h^{p-1}\rho^{k+1-q}}{\ell}+(p-1)fh^{p-2}\rho^{k+1-q}|\nabla u|^{2}+(\frac{uu_{m}\ell_{m}}{\ell^{2}}-\frac{|\nabla u|^{2}}{\ell})\sigma_{k}^{ii}\notag\\
&+\frac{p-k-1}{\ell^{2}}fh^{p-1}\rho^{k+1-q}u_{m}\ell_{m}+(\frac{2u_{m}\ell_{m}}{\ell^{2}}-\ga\frac{|\nabla u|^{2}}{u\ell})\sigma_{k}^{ij}u_{i}\ell_{j}\notag\\
&+\frac{(k+1-q)fh^{p-1}\rho^{k-q-1}(\ell u_{m}u_{r}b_{rm}+uu_{m}\ell_{r}b_{rm})}{\ell}\notag\\
&-\frac{2(1-\ell)}{\ell}\sigma_{k}^{ij}u_{i}u_{j}+\be\frac{|\nabla u|^{2}}{\ell^{2}}\sigma_{k}^{ij}\ell_{i}\ell_{j}.\label{12}
\end{align}
For the standard metric on spherical cap $\Ctheta$, we have the commutator formulas 
\begin{equation}\label{13}
u_{kij}=u_{ijk}+h_{k}\delta_{ij}-h_{j}\delta_{ki}.
\end{equation}
Using \eqref{13}, we derive
\begin{equation}\label{14}
\sigma_{k}^{ij}u_{m}u_{mij}=\sigma_{k}^{ij}u_{m}u_{ijm}+|\nabla u|^{2}\sigma_{k}^{ii}-\sigma_{k}^{ij}u_{i}u_{j}.
\end{equation}
Inserting \eqref{12} and \eqref{14} into \eqref{4}, we obtain
$$0\geq I+II,$$
where 
\begin{align}\label{21}
I:=\frac{2(k+1-q)}{\ell|\nabla u|^{2}}fh^{p-1}\rho^{k-q-1}(\ell u_{m}u_{r}b_{rm}+uu_{m}\ell_{r}b_{rm}),
\end{align}
and 
\begin{align}
II:=&\frac{2u_{m}f_{m}}{\ell|\nabla u|^2}h^{p-1}\rho^{k-q+1}+\frac{2}{|\nabla u|^2}\sigma_k^{ij}u_{mi}u_{mj}+\frac{2(p-1-k)}{\ell^2|\nabla u|^2}fh^{p-1}\rho^{k+1-q}u_m\ell_m\notag\\
&+\frac{\beta-\beta^2}{\ell^2}\sigma_{k}^{ij}\ell_i\ell_j+(2(p-1)-\gamma k)fh^{p-2}\rho^{k+1-q}+\Big(\frac{\gamma-\gamma^2}{u^2}-\frac{2(2-\ell)}{|\nabla u|^2\ell}\Big)\sigma_k^{ij}u_iu_j\notag\\
&+\Big(4\frac{u_m\ell_m}{\ell^2|\nabla u|^2}+\frac{2\beta\gamma}{u\ell}\Big)\sigma_k^{ij}u_i\ell_j+\Big(\frac{2u\sum_mu_m\ell_m}{|\nabla u|^2\ell^2}+(\beta-2)\frac{1-\ell}{\ell}+\frac{\gamma}{\ell}\Big)\sigma_k^{ii}.\label{15}
\end{align}
To derive a contradiction, we show that $I > 0$ and $II \geq 0$ when $\hat{C}$ is sufficiently large.

On the one hand, by \eqref{assume}, $\ell$ has a uniform bound on $\Ctheta$ and $|\nabla \ell| \leq C(\theta)$, hence
\begin{equation}\label{16}
\frac{u|\nabla\ell|}{\ell|\nabla u|}=O(\hat{C}^{-1/2}),\quad \frac{u|\nabla f|}{f|\nabla u|}=O(\hat{C}^{-1/2}),\quad \frac{u^{2}}{|\nabla u|^{2}}=O(\hat{C}^{-1}).
\end{equation}
By the Cauchy-Schwarz inequality
$$(\sum_{m}u_{m}u_{mi})^{2}\leq (\sum_{m}u_{m}^{2})(\sum_{m}u_{mi}^{2})=|\nabla u|^{2}\sum_{m}u_{mi}^{2},$$
which implies $\sum_{m}u_{mi}^{2}\geq \frac{(\sum_{m}u_{m}u_{mi})^{2}}{|\nabla u|^{2}}$. 
\begin{align}\label{17}
\frac{2}{|\nabla u|^{2}}\sigma_{k}^{ij}u_{mi}u_{mj}\geq \sum_{i}\frac{2\sigma_{k}^{ii}}{|\nabla u|^{2}}\frac{(\sum_{m}u_{m}u_{mi})^{2}}{|\nabla u|^{2}}.
\end{align}
In view of \eqref{16}, \eqref{17} and \eqref{15}, we get 
\begin{align*}
II&\geq fh^{p-2}\rho^{k+1-q}(2(p-1)-k\ga-\frac{c}{\sqrt{\hat{C}}})+(\ga-\frac{c}{\sqrt{\hat{C}}})\sigma_{k}^{ii}\notag\\
&-(\frac{c}{\sqrt{\hat{C}}}+2\ga-\ga^{2})\sigma_{k}^{ij}\frac{|u_{i}||\ell_{j}|}{u\ell}+(\ga-\frac{\ga^{2}}{2})\sigma_{k}^{ij}\frac{\ell_{i}\ell_{j}}{\ell^{2}}\notag\\
&+\frac{(\ga-\ga^{2}/2-\frac{c}{\hat{C}})}{u^{2}}{\sigma_{k}^{ij}}u_{i}u_{j}.
\end{align*}
Let $\ep \in (0,1)$. Choose $\hat{C}$ sufficiently large such that 
\begin{equation*}
\frac{c}{\sqrt{\hat{C}}}< \min\{\frac{1}{2}(2(p-1)-k\ga), \frac{\ga}{2}, 2(\sqrt{1+\ep}-1)a_{\ga}\},\quad \frac{c}{\hat{C}}<\ep a_{\ga}.
\end{equation*}
Hence,
\begin{align*}
II\geq& \frac{2(p-1)-k\ga}{2}fh^{p-2}\rho^{k+1-q}+\frac{\ga}{2}\sigma_{k}^{ii}-2\sqrt{1+\ep}a_{\ga}\sigma_{k}^{ij}\frac{|u_{i}||\ell_{j}|}{u\ell}\\
\quad&+a_{\ga}\sigma_{k}^{ij}\frac{\ell_{i}\ell_{j}}{\ell^{2}}+(1-\ep)a_{\ga}\sigma_{k}^{ij}\frac{u_{i}u_{j}}{u^{2}},
\end{align*}
where $a_{\ga}:=\ga-\frac{\ga^{2}}{2}>0$. Using the Cauchy-Schwarz inequality
$$2\sqrt{1+\ep}\frac{|u_{i}||\ell_{j}|}{u\ell}\leq \frac{1+\ep}{1-\ep}\frac{\ell_{i}^{2}}{\ell^{2}}+(1-\ep)\frac{u_{i}^{2}}{u^{2}},$$
we obtain
$$II\geq \frac{2(p-1)-{k\ga}}{2}fh^{p-2}\rho^{k+1-q}+\sigma_{k}^{ij}\Big(\frac{\ga}{2}-\frac{2\ep}{1-\ep}a_{\ga}\frac{\ell_{i}\ell_{j}}{\ell^{2}}\Big).$$
Note that $\ell_{i}^{2}\leq |\nabla \ell|^{2}\leq c(\theta)\ell^{2}$. Therefore, if $\ep>0$ is sufficiently small such that $$\frac{\ga}{2}-c(\theta)a_{\ga}\frac{2\ep}{1-\ep}>0,$$
which implies $II>0$. On the other hand,
\begin{align}
\ell\sum_{k,m}u_mu_kb_{km}&=\ell\sum_{k,m} u_mu_k(\ell u_{km}+u_k\ell_m+u_m\ell_k+u\delta_{km})\notag\\
&=\ell^2\sum_{k,m}{u_m u_k u_{km}}+2\ell |\nabla u|^2\sum_m{u_m\ell_m}+u\ell|\nabla u|^2\notag\\
&=\ell^2\sum_k u_k\frac{|\nabla u|^2}{2}\left(\gamma\frac{u_k}{u}-\beta\frac{\ell_k}{\ell}\right)+2\ell |\nabla u|^2\sum_m{u_m\ell_m}+u\ell|\nabla u|^2\notag\\
&=\frac{\gamma\ell^2}{2u}|\nabla u|^4+\left(2-\frac{\beta}{2}\right)\ell|\nabla u|^2\sum_{m}{u_m\ell_m}+u\ell|\nabla u|^2,\label{19}
\end{align}
where we used \eqref{2} in the third equality. Similarly, we have 
\begin{equation}\label{20}
uu_{m}\ell_{k}b_{km}=(\frac{|\nabla u|^{2}}{2}\ga\ell+u^{2})u_{m}\ell_{m}+u|\nabla u|^{2}|\nabla \ell|^{2}(1-\frac{\be}{2})+u(u_{m}\ell_{m})^{2}.
\end{equation}
Recalling assumption \eqref{assume}, and combining \eqref{19}, \eqref{20} with \eqref{21}, we have
\begin{align*}
\frac{\ell u_{m}u_{r}b_{rm}+uu_{m}\ell_{r}b_{rm}}{|\nabla u|^{2}}&\geq \left\{\frac{\ga}{2}\ell^{2}\frac{|\nabla u|}{u}-\Big((1+\ga)\ell+\frac{u^{2}}{|\nabla u|^{2}}\Big)|\nabla \ell|\right\}|\nabla u|\\
&\geq \left\{\frac{\ga}{2}\ell^{2}\sqrt{\hat{C}}-\Big((1+\ga)\ell+\frac{1}{\hat{C}}\Big)|\nabla \ell|\right\}|\nabla u|>0
\end{align*}
for $\hat{C}$ sufficiently large. The assumption \eqref{assume} does not hold. The proof is completed.
\end{proof}
The following proposition generalizes \cite{matheng}*{Lemma 3.1} to even capillary convex bodies. We refer the reader to \cite{HuYang2026John}*{Section 4} for the proof.
\begin{proposition}\label{noncollapsing}
Let $\theta\in (0,\frac{\pi}{2})$. Let $\Sigma$ be an even, smooth, strictly convex $\theta$-capillary hypersurface with capillary support function $h$. If $h$ satisfies
\begin{equation*}
      \frac{|\nabla h|^2}{h^{\gamma}} \leq N (\max_{\mathcal{C}_{\theta}} h)^{2-\gamma}, \quad \rm{on} \ \mathcal{C}_{\theta},
\end{equation*}
for some positive constants $\gamma$ and $N$. Then the following non-collapsing estimate holds
\begin{equation*}
    \frac{R}{r}\leq C,
\end{equation*}
where the constant $C$ depends only on $n,\theta,\gamma, N$.
\end{proposition}

Combining Lemma \ref{lem2} and Proposition \ref{noncollapsing}, we obtain the following uniform estimates for $h$.
\begin{lemma}\label{lem3}
Let $1<p<q\leq k+1$ and $\theta\in (0,\frac{\pi}{2})$. Suppose $h$ is an even, smooth and strictly convex solution to Eq. \eqref{eq:capillary-Lp-dual-CM}.
 Then there exists some positive constant $C$ depending on $p, q,\theta, f$ such that
\begin{equation}\label{C0}
\frac{1}{C}\leq h \leq C,
\end{equation}
and
\begin{equation}\label{C1}
|\nabla h|\leq C.
\end{equation}
\end{lemma}
\begin{proof}
In terms of the capillary support function $u$, \eqref{eq:capillary-Lp-dual-CM} becomes
\begin{equation*}\label{u}
		\renewcommand{\arraystretch}{1.5}
		\left\{
		\begin{array}{rll}
			\sigma_{k}(\ell\nabla^2 u+\nabla u\otimes \nabla \ell +\nabla \ell\otimes \nabla u+uI) &= f (u\ell)^{p-1}\rho(u,\ell)^{k+1-q}, &{\rm in} \ \mathcal{C}_{\theta}, \\
		\nabla_{\mu}u&=0,& \text{on }\partial\Ctheta.
			\end{array}	
			\right.
	\end{equation*}
Assume that $\underset{\Ctheta}{\max}\ u(\xi)$ is attained at $\xi_{0}$, then at $\xi_{0}$ we have
$$\nabla u=0,\quad \nabla^{2}u\leq 0.$$
Therefore, we obtain
$$C_{n}^{k}u^{k}\geq f(u\ell)^{p-1}u^{k+1-q}(\ell^{2}+|\nabla \ell|^{2})^{\frac{k+1-q}{2}},$$
which implies
$$\underset{\Ctheta}{\max}\ h\geq C (\underset{\Ctheta}{\min}\ f)^{\frac{1}{q-p}},$$
where $C$ is a positive constant that may change from line to line, and depends only on $n,k,p,q,\theta$.
Similarly, we also obtain
$$\underset{\Ctheta}{\min}\ h\leq C (\underset{\Ctheta}{\max}\ f)^{\frac{1}{q-p}}.$$
 If $p < q$, we can derive a uniform upper bound for $r$ and a uniform lower bound for $R$. These facts, together with Lemma \ref{lem2} and Proposition \ref{noncollapsing}, yield \eqref{C0}. Equation \eqref{C1} follows directly from \eqref{C0} and Lemma \ref{lem2}.
\end{proof}

	The following full rank theorem ensures the strict convexity of solutions to \eqref{eq:capillary-Lp-dual-CM} when $1\leq k<n$.
 \begin{theorem}\label{full rank}
 Let $1\leq k<n$, $p\geq1$, $q\in \R$ and $\theta \in (0,\frac{\pi}{2})$. Suppose $f$ is a positive smooth function satisfying
\begin{equation}\label{23}
\nabla^{2}{f}^{-\frac{1}{p+k-1}}+gf^{-\frac{1}{p+k-1}}\geq 0
\quad\rm{in }\ \Ctheta
\end{equation}
and the boundary condition 
\begin{equation}\label{24}
\ \ \ \nabla_{\mu}f^{-\frac{1}{p+k-1}}\leq \cot\theta f^{-\frac{1}{p+k-1}}\quad\rm{on }\ \pa\Ctheta.
\end{equation}
If $h$ is a positive and smooth solution of Eq. \eqref{eq:capillary-Lp-dual-CM} with $\nabla^{2}h+hI\geq 0$ on $\Ctheta$, then $\nabla^{2}h+hI$ is positive definite on $\Ctheta$.
 \end{theorem}
\begin{proof}
Define
$$F=\sigma_{k}^{\frac{1}{k}},\quad \tilde{f}=\hat{f}^{1/k}=(fh^{p-1}\rho^{k+1-q})^{1/k}.$$
Denote by $\lambda_{1}$ the smallest eigenvalue of $b[h]$. Suppose $\lambda_{1}=0$ somewhere in the interior of $\Ctheta$. By \cite{CabezasMorenoHu2025}*{Theorem 3.4, (3.20)} and condition \eqref{23}, there exists a positive constant $C$ depending on $k, p,q,||h||_{C^{3}(\Ctheta)}$ and $||f||_{C^{2}(\Ctheta)}$, such that
$$L[\lambda_{1}]:=F^{ij}\nabla^{2}_{ij}\lambda_{1}-C(\lambda_{1}+|\nabla \lambda_{1}|)\leq 0$$
in the viscosity sense. By the strong maximum principle, $\lambda_{1}\equiv0$ in $\Ctheta$. However, $\nabla_{\mu}h|_{\pa\Ctheta}=\cot\theta h>0$ implies that at the point where $h$ attains its minimum, we have $\lambda_{1}>0$. This is a contradiction.

Choose an orthonormal frame $\{e_{i}\}_{i=1}^{n}$ at $\xi_{0}$ such that
$$e_{1}=\mu,\quad e_{\al}\in T_{\xi_{0}}\pa \Ctheta.$$
We next assume that $\lambda_{1}(\xi_{0}) = 0$ for some $\xi_{0} \in \partial\Ctheta$ while $\lambda_{1} > 0$ in the interior. Then, by following the proof of \cite{HuIvakiScheuer2025}*{Theorem 3.1} with minor modifications, we obtain a contradiction. For completeness, we sketch the argument below. 

{\bf Step1}. We prove
$$b_{ii}(\xi_{0})=0\quad\Rightarrow\quad\nabla_{\mu}b_{ii}(\xi_{0})\geq 0.$$
By the Weingarten equation and Gaussian formula, we have
\begin{equation}\label{GW}
b_{\al\al,\mu}=\cot\theta(b_{\mu\mu}-b_{\al\al}),
\end{equation}
where we used the fact $h_{\al \mu}=0$.
For the case $i=\al\geq2$, this follows immediately from \eqref{GW}. For the case $i=1$, recall that \eqref{eq:capillary-Lp-dual-CM} is equivalent to
\begin{equation}\label{feq}
F(b[h])=\tilde{f}\quad \text{in}\ \Ctheta.
\end{equation}
Differentiating \eqref{feq} in the $\mu$-direction gives 
$$\sum_{i}F^{ii}b_{ii,\mu}=\tilde{f}_{\mu}.$$
By \eqref{GW}, we obtain 
\begin{equation}\label{25}
F^{\mu\mu}b_{\mu\mu,\mu}=\tilde{f}_{\mu}+\sum_{\al}F^{\al\al}\cot\theta(b_{\al\al}-b_{\mu\mu}).
\end{equation}
Note that $b_{\mu\mu}(\xi_{0})=0$, and by the $1$-homogeneity of $F$, we get
$$\sum_{\al}F^{\al\al}b_{\al\al}=F=\tilde{f}.$$
Evaluating \eqref{25} at $\xi_{0}$ yields
$$b_{\mu\mu,\mu}=\frac{\tilde{f}_{\mu}+\tilde{f}\cot\theta}{F^{\mu\mu}}.$$
Thus,
$$
\frac{\tilde{f}_{\mu}+\tilde{f}\cot\theta}{F^{\mu\mu}}\geq0\Leftrightarrow \nabla_{\mu}(\log \tilde{f})\geq -\cot\theta.
$$
Using properties (1) and (2) in Lemma \ref{lem1}, we require that
\begin{align*}
&\frac{1}{k}\nabla_{\mu}\Big(\log f+(p-1)\log h+(k+1-q)\log \rho\Big)\\
&=\frac{1}{k}\Big(\frac{\nabla_{\mu}f}{f}+(p-1)\cot\theta+(k+1-q)\frac{\cot\theta hb_{\mu\mu}}{\rho^{2}}\Big)\\
&=\frac{1}{k}\Big(\frac{\nabla_{\mu}f}{f}+(p-1)\cot\theta\Big).
\end{align*}
Therefore
$$\nabla_{\mu}(\log \tilde{f})\geq -\cot\theta\Leftrightarrow \nabla_{\mu}\log f\geq -(k+1-p)\cot\theta\quad \text{on}\ \pa\Ctheta,$$
which is precisely \eqref{24}.

{\bf Step2}. Consider an interior ball $B_{\tilde{R}}(x_0)\subset\Ctheta$ touching at $\xi_{0}$. Define an annular region $A_{\tilde{R},\tilde{\rho}}=B_{\tilde{R}}(x_0)\setminus  \operatorname{int}(B_{\tilde{\rho}}(x_0))$ for some $0<\tilde{\rho}<\tilde{R}$.
For $x\in \Ctheta$, let $\tilde{r}(x)=\operatorname{dist}(x,x_0)$ denote the distance of $x$ to $x_0$. We define
$$w(x)=e^{-\al \tilde{R}^2}-e^{-\al \tilde{r}(x)^2}.$$
For any $x\in A_{\tilde{R},\tilde{\rho}}$, the distance function $\tilde{r}(x)$ satisfies 
$$
\nabla^2_{ij}\tilde{r}(x)=\cot \tilde{r}(x) (g_{ij}-\nabla_{i}\tilde{r} \nabla_{j} \tilde{r}).
$$
Therefore,
\begin{align*}
\nabla_{i} w &=2\al \tilde{r} \nabla_{i}\tilde{ r} e^{-\al \tilde{r}^2},\\
\nabla^2_{i,j} w&=-4\al^2 \tilde{r}^2 \nabla_{i} \tilde{r}  \nabla_{j} \tilde{r} e^{-\al \tilde{r}^2}+2\al \nabla_{i}\tilde{r} \nabla_{j} \tilde{r} e^{-\al \tilde{r}^2}+2\al \tilde{r} \nabla^2_{i,j} \tilde{r} e^{-\al \tilde{r}^2} \\
&=e^{-\al \tilde{r}^2}\Big( (-4\al^2 \tilde{r}^2+2\al-2\al \tilde{r} \cot\tilde{r}) \nabla_{i} \tilde{r} \nabla_{j} \tilde{r}+2\al \tilde{r} \cot \tilde{r} g_{ij}\Big)
\end{align*}	
and
\begin{align*}
L[w]= &~e^{-\al \tilde{r}^2}(-4\al^2 \tilde{r}^2+2\al-2\al \tilde{r}\cot \tilde{r})|\nabla \tilde{r}|_{\dot{F}}^2\\
 &~+2\al \tilde{r}\cot \tilde{r} e^{-\al \tilde{r}^2}\operatorname{tr}(\dot{F})-2 c \al \tilde{r} e^{-\al \tilde{r}^2}+c(e^{-\al \tilde{r}^2}-e^{-\al \tilde{R}^2}).
\end{align*}

Assume that $\lambda \delta_{ij} \leq F^{ij} \leq \Lambda \delta_{ij}$ in $B_{\tilde{R}}(x_0)$, where $0<\lambda\leq \Lambda$ are some constants. Note that $0\leq \tilde{r}\cot \tilde{r} \leq 1$ for any $\tilde{\rho} \leq \tilde{r}\leq \tilde{R}$. Now take $\al>0$ sufficiently large so that in $A_{\tilde{R},\tilde{\rho}}$, 
\begin{equation*}
L[w] \leq -e^{-\al \tilde{r}^2}\left[ (4\al^2\tilde{r}^2-2\al)\lambda -2\al n  \Lambda +2c\al \tilde{r}-c\right]-ce^{-\al \tilde{R}^2}<0.
\end{equation*}

{\bf Step3}: Since $\lambda_1(x)>0$ on $\partial B_{\tilde{\rho}}(x_0)$, there exists $\varepsilon>0$ such that 
$$
\psi(x):=\lambda_1(x)+\varepsilon w(x) > 0
$$
on $\partial B_{\tilde{\rho}}(x_0)$. Note that $w = 0$ on $\partial B_{\tilde{R}}(x_0)$ and by our assumption there is no other point on $\partial B_{\tilde{R}}(x_0) \setminus \{\xi_{0}\}$ where $\lambda_1 = 0$; hence, $\psi > 0$ on $\partial B_{\tilde{R}}(x_0) \setminus \{\xi_{0}\}$. By the maximum principle for the viscosity supersolution $\psi$, we have $\psi\geq 0$ in the annulus, where $\psi$ is also the smallest eigenvalue of
$$
S_{ij}=b_{ij}+\varepsilon w g_{ij}.
$$
Suppose that at $\xi_{0}$ the zero eigenvalue of $S(\xi_{0})$ is attained in direction $e_{i}$, i.e., we also have $b_{ii} = 0$.
Let $\ga$ be a unit speed geodesic in direction $-\mu$ and $e_i$ be parallel transported along $\ga$. Then from {\bf Step 1} we get
$$
0\leq \frac{d}{dt}\Big|_{t=0^+}S(\gamma(t))(e_i,e_i)=\nabla_{-\mu}b_{ii}+\varepsilon  \nabla_{-\mu}w\leq-2\alpha \varepsilon \tilde{R} <0,
$$
a contradiction.
\end{proof}
Inspired by \cite{Gcpam,jfa}, we consider the following general Christoffel--Minkowski type equation with Robin boundary condition:
\begin{equation}\label{general}
		\renewcommand{\arraystretch}{1.5}
		\left\{
		\begin{array}{rll}
			\sigma_{k}\big(\nabla^2h+hI\big)&=\hat{f}(\xi, h, \nabla h),& \text{in }\Ctheta,\\
		\nabla_{\mu}h&=\cot\theta h,& \text{on }\partial\Ctheta.
			\end{array}	
			\right.
	\end{equation}
Let $N:=(0,\cdots,0,1-\cos\theta) \in \Ctheta$. We introduce the function $d:=\frac{1}{2\theta}d_{N}^{2}(\xi)$, where $d_{N}(\xi)$ is geodesic distance function from $\xi$ to $N$ on $\Ctheta$. Clearly $d$ is well-defined and smooth for all $\xi \in \Ctheta$. A direct computation shows that $d$ satisfies $d=\frac{\theta}{2}$, $\nabla d=\mu$ on $\pa\Ctheta$ and 
$$(\nabla^{2}_{ij}d)\geq \min \left\{\frac{1}{\theta}, \cot\theta\right\}\delta_{ij},\quad \zeta_{\mu}|_{\pa\Ctheta}=e^{-\frac{\theta}{2}}.$$
Using the distance function, we construct a boundary barrier term $\mathcal{B}:=\langle \nabla h, \nabla \zeta \rangle$ with $\zeta:=1-e^{-d}$. By choosing suitable test functions and applying the maximum principle both on the boundary and in the interior, we obtain a global $C^{2}$ estimate.
\begin{theorem}\label{thm2}
Let $1\leq k\leq n$. Suppose $h$ is a positive, smooth and strictly convex solution to Eq. \eqref{general}. Then there exists a positive $C$ such that
$$\Delta h+nh\leq C,$$
where $C$ depends on $n,k,\underset{\Ctheta}{\min}\ h, ||h||_{C^{1}(\Ctheta)}, \underset{\Ctheta}{\min} \hat{f}$ and $||\hat{f}||_{C^{2}(\Ctheta)}$.
\end{theorem}
\begin{proof}
We consider the auxiliary function 
\begin{equation}\label{test}
\Phi=\frac{1}{m}\log P_{m}+\frac{A}{2}\ell^{2}|\nabla u|^{2}-K\mathcal{B}+M\log h-L\zeta,
\end{equation}
where $A, K, -M, L$ are positive constants to be determined later, and 
$$P_{m}= \sum_{j}\lambda_{j}^{m}, \quad m\geq 2.$$
Here $\lambda_{1}, \lambda_{2},\cdots, \lambda_{n}$ are the eigenvalues of the spherical Hessian $\nabla^{2}h+hI$.
Assume that $\Phi$ attains its maximum at some point, say $\xi_{0} \in \Ctheta$. We divide the proof into two cases: either $\xi_{0} \in \pa\Ctheta$ or $\xi_{0} \in \Ctheta\setminus\pa\Ctheta$.

{\bf Case I}. $\xi_{0} \in \pa\Ctheta$. Let $\{e_{i}\}_{i=1}^{n}$ be an orthonormal frame around $\xi_{0}$ such that $e_{n}=\mu$ and $b[h]$ is diagonal. First, we have
\begin{equation*}
\frac{1}{m}\nabla_{\mu}\log P_{m}=\frac{\sum_{\al}^{n-1}\lambda_{\al}^{m-1}b_{\al\al,\mu}+\lambda_{\mu}^{m-1}b_{\mu\mu,\mu}}{P_{m}}.
\end{equation*}
We differentiate \eqref{general} and deduce
\begin{equation}\label{c22}
\hat{f}_{\mu}\leq C_{1}(1+b_{\mu\mu}),
\end{equation}
where we used again the part $(2)$ in Lemma \ref{lem1}.
We also have
\begin{equation}\label{c23}
\sigma_{k}^{\mu\mu}b_{\mu\mu,\mu}+\sum_{\al}^{n-1}\sigma_{k}^{\al\al}b_{\al\al,\mu}=\hat{f}_{\mu}.
\end{equation}
Combining with \eqref{c22} and \eqref{c23}, we have
\begin{align*}
\nabla_{\mu} \frac{1}{m}\log P_{m}&\leq \frac{\cot\theta(b_{\mu\mu}-b_{\al\al})\Big(\lambda_{\al}^{m-1}-\lambda_{\mu}^{m-1}\frac{\sigma_{k}^{\al\al}}{\sigma_{k}^{\mu\mu}}\Big)}{P_{m}}+\frac{\lambda_{\mu}^{m-1}\hat{f}C_{1}(1+b_{\mu\mu})}{P_{m}\sigma_{k}^{\mu\mu}}\\
&:=T_{1}+T_{2}.
\end{align*}
On the one hand, we show that $T_{1}\leq 0$ and $T_{2}\leq C_{2}(1+b_{\mu\mu})$. If $\lambda_{\al}\geq \lambda_{\mu}$, we have
$$\sigma_{k}^{\al\al}=\sigma_{k-1}(\lambda|\al)\leq \sigma_{k-1}(\lambda|\mu)=\sigma_{k}^{nn}\quad\Rightarrow \frac{\sigma_{k}^{\al\al}}{\sigma_{k}^{nn}}\leq 1.$$
Then 
$$T_{1}:=\frac{\cot\theta(b_{\mu\mu}-b_{\al\al})\Big(\lambda_{\al}^{m-1}-\lambda_{\mu}^{m-1}\frac{\sigma_{k}^{\al\al}}{\sigma_{k}^{\mu\mu}}\Big)}{P_{m}}\leq 0.$$
The same discussion also applies to the case $\lambda_{\mu} \geq \lambda_{\alpha}$.
On the other hand, from Proposition \ref{basic} and Proposition \ref{NM}, we have
\begin{equation}\label{c24}
\hat{f}=\sigma_{k}=\lambda_{\mu}\sigma_{k-1}(\lambda|\mu)+\sigma_{k}(\lambda|\mu),
\end{equation}
and 
\begin{equation}\label{c25}
\left[\frac{\sigma_{k-1}(\lambda|\mu)}{\binom{n-1}{k-1}}\right]^{\frac{1}{k-1}}\geq \left[\frac{\sigma_{k}(\lambda|\mu)}{\binom{n-1}{k}}\right]^{\frac{1}{k}}.
\end{equation}
Using \eqref{c25}, for some positive constant $C_{n,k}$, depending only on $n,k$, we get
$$\sigma_{k}(\lambda|\mu)\leq C_{n,k}\sigma_{k-1}(\lambda|\mu)\underset{i}{\max}\ \lambda_{i}.$$
So by \eqref{c24}, we have
$$\frac{\hat{f}}{\sigma_{k}^{\mu\mu}}\leq C_{3}\underset{i}{\max}\ \lambda_{i}.$$
Then
$$T_{2}:=\frac{\lambda_{\mu}^{m-1}\hat{f}C_{1}(1+b_{\mu\mu})}{P_{m}\sigma_{k}^{\mu\mu}}\leq\frac{C_{3}\underset{i}{\max}\ \lambda_{i}^{m}C_{1}(1+b_{\mu\mu})}{\underset{i}{\max}\ \lambda_{i}^{m}}:=C_{4}(1+b_{\mu\mu}).$$
At $\xi_{0}$, the maximum value condition implies
\begin{align*}
0\leq \nabla_{\mu}\Phi\leq& C_{4}(1+b_{\mu\mu})+A\ell\nabla_{\mu}\ell|\nabla u|^{2}+A\ell^{2}u_{k}u_{k\mu}-Ke^{-\frac{\theta}{2}}b_{\mu\mu}+KC_{5}+M\cot\theta-Le^{-\frac{\theta}{2}}\\
\leq & b_{\mu\mu}(C_{4}-Ke^{-\frac{\theta}{2}})+(C_{4}+KC_5{}-Le^{-\frac{\theta}{2}}),
\end{align*}
where we used Lemma \ref{lem1}. Choosing $L\gg K\gg1$ yields a contradiction. 

{\bf Case II}. $\xi_{0} \in \Ctheta\setminus\pa\Ctheta$. Hence at $\xi_{0}$, we have
\begin{equation}\label{c25-foc}
\frac{\sum_{j}\lambda_{j}^{m-1}b_{jj,i}}{P_{m}}+A\ell|\nabla u|^{2}\ell_{i}+A\ell^{2}u_{k}u_{ki}-K(h_{ki}\zeta_{k}+h_{k}\zeta_{ki})+M\frac{h_{i}}{h}-L\zeta_{i}=0,
\end{equation}
and
\begin{equation}
\begin{split}
\label{bjh}
0\geq &\frac{1}{P_{m}}\left(\sum_{j}\lambda^{m-1}_{j}b_{jj;ii}+(m-1)\sum_{j}\lambda^{m-2}_{j}b^{2}_{jj;i}+\sum_{p,q;p\neq q}\frac{\lambda^{m-1}_{p}-\lambda^{m-1}_{q}}{\lambda_{p}-\lambda_{q}}b^{2}_{pq;i}  \right)\\
&\quad -\frac{m}{P^{2}_{m}}(\sum_{j}\lambda^{m-1}_{j}b_{jj;i})^{2}+M\frac{h_{ii}}{h}-M\frac{h^{2}_{i}}{h^{2}}-L\zeta_{ii}+A\ell^{2}u_{ki}u_{ki}+A\ell^{2}u_{k}u_{kii}\\\
&+\ell_{i}^{2}A|\nabla u|^{2}+4A\ell u_{k}\ell_{i}u_{ki}-K(h_{kii}\zeta_{k}+2h_{ki}\zeta_{ki}+h_{k}\zeta_{kii})-A\ell^{2}|\nabla u|^{2}.\ 
\end{split}
\end{equation}
 Differentiating \eqref{general} twice, at $\xi_{0}$, we obtain
\begin{equation}\label{bjk}
\sum_{i}\sigma^{ii}_{k}b_{ii;j}=\hat{f}_{h_{j}}h_{jj}+\hat{f}_{h}h_{j}+\hat{f}_{j},
\end{equation}
and
\begin{equation}
\begin{split}
\label{bjl}
\sum_{i}\sigma^{ii}_{k}b_{ii;jj}+\sum_{p,q,r,s}\sigma^{pq,rs}_{k}b_{pq;j}b_{rs;j}&\geq  -C-C |h_{jj}|-C h^{2}_{jj}+\sum_{s}\hat{f}_{h_{s}}h_{sj;j}\\
&\geq -C-Cb_{jj}-Cb^{2}_{jj}+\sum_{s}\hat{f}_{h_{s}}b_{sj;j}.
\end{split}
\end{equation}
Here, the positive constants $C$ depend on $\underset{\Ctheta}{\min}h, ||h||_{C^{1}(\Ctheta)}, \underset{\Ctheta}{\min}\hat{f}$ and $||\hat{f}||_{C^{2}(\Ctheta)}$.

Now using the Ricci identity $b_{ii;jj}=b_{jj;ii}+b_{ii}-b_{jj}$, we obtain
\begin{equation}\label{bjp}
\sum_{i}\sigma^{ii}_{k}b_{jj;ii}=\sum_{i}\sigma^{ii}_{k}b_{ii;jj}+b_{jj}\sum_{i}\sigma^{ii}_{k}-k\hat{f}.
\end{equation}
Multiplying both sides of \eqref{bjh} by $\sigma^{ii}_{k}$ and using \eqref{bjk}, \eqref{bjl}, and \eqref{bjp}, we obtain

\begin{equation}
\begin{split}
\label{zx}
0&\geq \frac{1}{P_{m}}\sum_{j}\lambda^{m-1}_{j}\Big(-C_{7}-C_{8}b_{jj}-C_{9}b^{2}_{jj}+\sum_{s}\hat{f}_{h_{s}}b_{sj;j}-W(\sigma_{k})_{j}^{2}+W(\sigma_{k})_{j}^{2}\\
&+b_{jj}\sum_{i}\sigma^{ii}_{k}-k\hat{f}-\sum_{p,q,r,s}\sigma^{pq,rs}_{k}b_{pq;j}b_{rs;j}\Big)\\
&\quad +\frac{1}{P_{m}}(m-1)\sum_{i}\sigma^{ii}_{k}\sum_{j}\lambda^{m-2}_{j}b^{2}_{jj;i}+\frac{1}{P_{m}}\sum_{i}\sigma^{ii}_{k}\sum_{p,q;p\neq q}\frac{\lambda^{m-1}_{p}-\lambda^{m-1}_{q}}{\lambda_{p}-\lambda_{q}}b^{2}_{pq;i}\\
&\quad-\sum_{i}\frac{m\sigma^{ii}_{k}}{P^{2}_{m}}(\sum_{j}\lambda^{m-1}_{j}b_{jj;i})^{2}
+M\sum_{i}\sigma^{ii}_{k}\frac{h_{ii}}{h}-M\frac{\sum_{i}\sigma^{ii}_{k}h^{2}_{i}}{h^{2}}-L\sigma_{k}^{ii}\zeta_{ii}\\
&+A\sum_{i}\sigma_{k}^{ii}(\ell_{i}^{2}|\nabla u|^{2}+4\ell\ell_{i}u_{k}u_{ki}-\ell^{2}|\nabla u|^{2})+A\sum_{i}\sigma_{k}^{ii}(\ell^{2}u_{ki}u_{ki})+A\ell^{2}\sum_{i}\sigma_{k}^{ii}u_{k}u_{kii}\\
&-K\sum_{i}\sigma_{k}^{ii}(h_{kii}\zeta_{k}+2h_{ki}\zeta_{ki}+h_{k}\zeta_{kii}).
\end{split}
\end{equation}
We now need to deal with the two terms $\frac{\sum_{s,j}\lbd_{j}^{m-1}\hat{f}_{h_ {s}}b_{sj,j}}{P_{m}}$ and $A\ell^{2}\sum_{i}\sigma_{k}^{ii}u_{k}u_{kii}$ in \eqref{zx}. Combining with \eqref{c25-foc} and the fact
$$b_{ij}=\ell u_{ij}+\ell_{i}u_{j}+\ell_{j}u_{i}+u\delta_{ij},$$
then
\begin{align}\label{c27}
&\frac{1}{P_{m}}\sum_{s,j}\lbd_{j}^{m-1}\hat{f}_{h_{s}}b_{sj,j}+A\ell^{2}\sum_{i}\sigma_{k}^{ii}u_{k}u_{kii}\notag\\
&=\sum_{s}\hat{f}_{h_{s}}\Big(-A\ell|\nabla u|^{2}\ell_{s}-A\ell^{2}u_{k}u_{ks}+K(h_{ks}\zeta_{k}+h_{k}\zeta_{ks})-M\frac{h_{s}}{h}+L\zeta_{s}\Big)\notag\\
&\quad\quad+A\sum_{i}\sigma_{k}^{ii}u_{k}\ell((\ell u_{ii})_{k}-\ell_{k}u_{ii})\notag\\
&= \sum_{s}\hat{f}_{h_{s}}\Big(-A\ell|\nabla u|^{2}\ell_{s}-A\ell^{2}u_{k}u_{ks}+K(h_{ks}\zeta_{k}+h_{k}\zeta_{ks})-M\frac{h_{s}}{h}+L\zeta_{s}\Big)\notag\\
&\quad\quad+A\sum_{i}\sigma_{k}^{ii}u_{k}\ell(b_{ii,k}-2\ell_{ik}u_{i}-2\ell_{i}u_{ik}-u_{k}-\ell_{k}u_{ii})\notag\\
&= \sum_{s}\hat{f}_{h_{s}}\Big(-A\ell|\nabla u|^{2}\ell_{s}-A\ell^{2}u_{k}u_{ks}+K(h_{ks}\zeta_{k}+h_{k}\zeta_{ks})-M\frac{h_{s}}{h}+L\zeta_{s}\Big)\notag\\
&\quad\quad+Au_{k}\ell\Big(\hat{f}_{k}+\hat{f}_{h}h_{k}+\hat{f}_{h_{s}}(\ell u_{sk}+u_{s}\ell_{k}+u_{k}\ell_{s}+u\delta_{sk}-\ell u\delta_{sk})\Big)-2A\ell \sum_{i}\sigma_{k}^{ii}u_{k}\ell_{ik}u_{i}\notag\\
&\quad\quad-2A\sum_{i}\sigma_{k}^{ii}u_{k}\ell_{i}(\ell u_{ik})-A\ell\sum_{i}\sigma_{k}^{ii}|\nabla u|^{2}-A\sum_{i}\sigma_{k}^{ii}u_{k}\ell_{k}(\ell u_{ii})\notag\\
&\quad\quad\geq -KC_{10}-AC_{11}-LC_{12}-MC_{13}-AC_{14}\sigma_{k}^{ii}-KC_{15}\underset{i}{\max}\lbd_{i}.
\end{align}
On the other hand, there holds
\begin{equation}\label{zr}
-\sum_{p,q,r,s}\sigma^{pq,rs}_{k}b_{pq;j}b_{rs;j}=-\sum_{p,q}\sigma^{pp,qq}_{k}b_{pp;j}b_{qq;j}+\sum_{p,q}\sigma^{pp,qq}_{k}b^{2}_{pq;j}.
\end{equation}
Then substituting \eqref{c27} and \eqref{zr} into \eqref{zx} and using \eqref{bjk}, \eqref{zx} becomes 
\begin{equation}
\begin{split}
\label{c26}
0&\geq \frac{1}{P_{m}}\sum_{j}\lambda^{m-1}_{j}\Big(-C_{0}(W)-C_{1}(W)b_{jj}-C_{2}(W)b^{2}_{jj}+W(\sigma_{k})_{j}^{2}\\
&-\sum_{p,q}\sigma^{pp,qq}_{k}b_{pp;j}b_{qq;j}+\sum_{p,q}\sigma^{pp,qq}_{k}b^{2}_{pq;j}\Big)+\frac{1}{P_{m}}(m-1)\sum_{i}\sigma^{ii}_{k}\sum_{j}\lambda^{m-2}_{j}b^{2}_{jj;i}\\
&\quad+\frac{1}{P_{m}}\sum_{i}\sigma^{ii}_{k}\sum_{p,q;p\neq q}\frac{\lambda^{m-1}_{p}-\lambda^{m-1}_{q}}{\lambda_{p}-\lambda_{q}}b^{2}_{pq;i}-\sum_{i}\frac{m\sigma^{ii}_{k}}{P^{2}_{m}}(\sum_{j}\lambda^{m-1}_{j}b_{jj;i})^{2}\\
&\quad +A\sum_{i}\sigma_{k}^{ii}(\ell^{2}u_{ki}u_{ki})+(-C_{16}K-C_{17}L-C_{18}A-M-M\frac{h_{i}^{2}}{h})\sigma_{k}^{ii}\\
&\quad-\tilde{C}-C_{19}K\underset{i}{\max}\lbd_{i},
\end{split}
\end{equation}
where the positive constant $\tilde{C}$ depends on $\underset{\Ctheta}{\min}h, ||h||_{C^{1}(\Ctheta)}, \underset{\Ctheta}{\min}\hat{f}, ||\hat{f}||_{C^{2}(\Ctheta)}$ and constants $M,A,K,L$. Moreover, the constants $C_{0}(W), C_{1}(W), C_{2}(W)$ depend on $\underset{\Ctheta}{\min}h, ||h||_{C^{1}(\Ctheta)}, \underset{\Ctheta}{\min}\hat{f}, ||\hat{f}||_{C^{2}(\Ctheta)}$ and positive constant $W$.

For the term $A\sum_{i}\sigma_{k}^{ii}(\ell^{2}u_{ki}u_{ki})$ in \eqref{c26}, we derive
\begin{align*}
A\sum_{i}\sigma_{k}^{ii}(\ell^{2}u_{ki}u_{ki})&=\sum_{i}\sigma_{k}^{ii}A(b_{ii}-2u_{i}\ell_{i}-u)^{2}\\
&\geq A\sum_{i}\sigma_{k}^{ii}(b_{ii}^{2}-4b_{ii}\ell_{i}u_{i}+4uu_{i}\ell_{i}-2ub_{ii})\\
&\geq A\sum_{i}\sigma_{k}^{ii}b_{ii}^{2}-C_{20}A-AC_{21}\sigma_{k}^{ii}.
\end{align*}
Now, if we take $-M\gg A\gg L\gg K\gg1$ while keeping the sign of the constant term $\tilde{C}$ unchanged, \eqref{c26} becomes 
\begin{equation}
\begin{split}
\label{c27-final}
0&\geq \frac{1}{P_{m}}\sum_{j}\lambda^{m-1}_{j}(-C_{0}(W)-C_{1}(W)b_{jj}-C_{2}(W)b^{2}_{jj}+W(\sigma_{k})^{2}_{j}\\
&-\sum_{p,q}\sigma^{pp,qq}_{k}b_{pp;j}b_{qq;j}+\sum_{p,q}\sigma^{pp,qq}_{k}b^{2}_{pq;j})\\
&\quad+\frac{1}{P_{m}}(m-1)\sum_{i}\sigma^{ii}_{k}\sum_{j}\lambda^{m-2}_{j}b^{2}_{jj;i}+\frac{1}{P_{m}}\sum_{i}\sigma^{ii}_{k}\sum_{p,q;p\neq q}\frac{\lambda^{m-1}_{p}-\lambda^{m-1}_{q}}{\lambda_{p}-\lambda_{q}}b^{2}_{pq;i}\\
&\quad-\sum_{i}\frac{m\sigma^{ii}_{k}}{P^{2}_{m}}(\sum_{j}\lambda^{m-1}_{j}b_{jj;i})^{2}+A\sum_{i}\sigma^{ii}_{k}b^{2}_{ii}-\tilde{C}-C_{19}K\lbd_{1},
\end{split}
\end{equation}
where $\lbd_{1}:=\underset{i}{\max}\lbd_{i}$.
Next we deal with the third-order derivatives. Denote
$$A_{i}=\frac{\lambda^{m-1}_{i}}{P_{m}}(W(\sigma_{k})^{2}_{i}-\sum_{p,q}\sigma^{pp,qq}_{k}b_{pp;i}b_{qq;i}),$$
$$B_{i}=\frac{2}{P_{m}}\sum_{j}\lambda^{m-1}_{j}\sigma^{jj,ii}_{k}b^{2}_{jj;i},$$
$$C_{i}=\frac{m-1}{P_{m}}\sigma^{ii}_{k}\sum_{j}\lambda^{m-2}_{j}b^{2}_{jj;i},$$
$$D_{i}=\frac{2}{P_{m}}\sum_{j\neq i}\sigma^{jj}_{k}\frac{\lambda^{m-1}_{j}-\lambda^{m-1}_{i}}{\lambda_{j}-\lambda_{i}}b^{2}_{jj;i},$$
$$E_{i}=\frac{m\sigma^{ii}_{k}}{P^{2}_{m}}(\sum_{j}\lambda^{m-1}_{j}b_{jj;i})^{2}.$$

We now turn to the application of a crucial lemma from \cite{jfa}*{Lemma 8, Lemma 9} (see also \cite{CabezasMorenoHu2025}*{Lemma 3.6, Lemma 3.7}) to estimate the third-order derivative terms. In particular, when $i=1$ and $i \neq 1$, the handling of the third-order derivatives follows the strategy developed in \cite{Gcpam}.
\begin{lemma}\label{h1}
For any $i\neq 1$, we obtain
\[
A_{i}+B_{i}+C_{i}+D_{i}-\left( 1+\frac{1}{m} \right)E_{i}\geq 0,
\]
for sufficiently large $m$.
\end{lemma}

\begin{lemma}\label{h2}
For $l=1,\ldots,k-1$. if there exist a positive constant $\delta\leq 1$ such that $\lambda_{l}/\lambda_{1}\geq \delta$.
Then there exist two sufficiently small positive constants $\eta,\delta^{'}$ depending on $\delta$, such that, if $\lambda_{l+1}/\lambda_{1}\leq \delta^{'}$, we obtain
$$
A_{1}+B_{1}+C_{1}+D_{1}-\left(1+\frac{\eta}{m} \right)E_{1}\geq 0,
$$
for sufficiently large $m$.
\end{lemma}
\end{proof}
Similar to \cite[Corollary 10]{jfa}, we have the following result.
\begin{lemma}\label{h3}
There exist two finite sequences of positive numbers $\{\delta_{j}\}^{k}_{j=1}$ and $\{\xi_{j}\}^{k}_{j=1}$ such that, if the following inequality holds for some index $1\leq s\leq k-1$,
$$
\frac{\lambda_{s}}{\lambda_{1}}\geq \delta_{s}, \ {\rm and} \ \frac{\lambda_{s+1}}{\lambda_{1}}\leq \delta_{s+1},
$$
then for large $W$, there holds
\begin{equation}\label{h4}
A_{1}+B_{1}+C_{1}+D_{1}-\left(1+\frac{\xi_{s}}{m} \right)E_{1}\geq 0.
\end{equation}
\end{lemma}
\begin{proof}
We proceed by induction to construct the sequences $\{\delta_{j}\}_{j=1}^{k}$ and $\{\xi_{j}\}_{j=1}^{k}$. To initiate the induction, we set $\delta_{1} = 1/2$, which clearly satisfies $\lambda_{1}/\lambda_{1} = 1 > \delta_{1}$. Consequently, the assertion for $j=1$ follows directly from Lemma \ref{h2}.

Suppose now that $\delta_{s}$ has been defined for all $1 \leq s \leq k-1$. In order to determine $\delta_{s+1}$, we apply Lemma \ref{h2} with parameters $l = s$ and $\delta = \delta_{s}$. This yields a positive constant $\delta'_{s+1}$ such that whenever $\lambda_{s+1} \leq \delta'_{s+1} \lambda_{1}$, the estimate \eqref{h4} holds for some $\xi_{s}$. We then choose $\delta_{s+1} := \min\{\delta_{1}, \delta'_{s+1}\}$, ensuring that \eqref{h4} is valid under the condition $\lambda_{s+1} \leq \delta_{s+1} \lambda_{1}$. Hence both $\delta_{s+1}$ and $\xi_{s}$ are determined inductively.
\end{proof}
\begin{proof}[\bf Proof of Theorem \ref{thm2}.]
To establish the desired result, we distinguish two cases.

{\bf Case I:} Suppose there exists an index $1 \leq s \leq k-1$ and positive constants $\{\delta_{j}\}_{j=1}^{k}$ such that
$$
\lambda_{s} \geq \delta_{s} \lambda_{1} \quad \text{and} \quad \lambda_{s+1} \leq \delta_{s+1} \lambda_{1}.
$$
Then, by combining Lemma \ref{h1} with Lemma \ref{h3}, we obtain
\begin{equation}\label{Pi}
\sum_{i}(A_{i}+B_{i}+C_{i}+D_{i})-E_{1}-\left(1+\frac{1}{m} \right)\sum^{n}_{i=2}E_{i}\geq 0.
\end{equation}

Recalling the definitions of $A_{i}$, $B_{i}$, $C_{i}$, $D_{i}$, and $E_{i}$, and substituting \eqref{Pi} into \eqref{c27-final}, we arrive at
\begin{equation}
\begin{split}
\label{Qi}
0 &\geq \frac{1}{P_{m}}\sum_{j}\lambda^{m-1}_{j}\left(-C_{0}(W)-C_{1}(W)b_{jj}-C_{2}(W)b^{2}_{jj}\right) \\
&\quad + \sum^{n}_{i=2}\frac{\sigma^{ii}_{k}}{P^{2}_{m}}\left(\sum_{j}\lambda^{m-1}_{j}b_{jj;i}\right)^{2} + A\sum_{i}\sigma^{ii}_{k}b^{2}_{ii} - \tilde{C}-C_{19}K\lbd_{1},\\
&\geq -\frac{\tilde{C}_{0}(W)}{\lambda_{1}} - \tilde{C}_{1}(W) - \tilde{C} - \tilde{C}_{2}(W)\lambda_{1} + A\sigma^{11}_{k}b^{2}_{11}-C_{19}K\lbd_{1}.
\end{split}
\end{equation}
Applying the Newton--MacLaurin inequality once again yields
\begin{equation}\label{Mac}
\left[\frac{\sigma_{k-1}(\lambda|1)}{\binom{n-1}{k-1}}\right]^{\frac{1}{k-1}}
\geq
\left[\frac{\sigma_{k}(\lambda|1)}{\binom{n-1}{k}}\right]^{\frac{1}{k}}.
\end{equation}

Consequently, from \eqref{Mac}, there exists a positive constant $C_{n,k}$, depending only on $n$ and $k$, such that
\begin{equation}
\label{c28}
\sigma_{k}(\lambda|1)
\leq C_{n,k} \sigma_{k-1}(\lambda|1)^{\frac{k}{k-1}}
\leq C_{n,k} \lambda_{1} \sigma_{k-1}(\lambda|1),
\end{equation}
where the last inequality follows from $\lambda_1 \geq \lambda_j$ for all $j$.

Substituting \eqref{c28} into the decomposition $\sigma_{k}(\lambda)=\sigma_{k}(\lambda|1)+\lambda_{1}\sigma_{k-1}(\lambda|1)$, we obtain
\begin{equation*}
\lambda_{1}\sigma_{k-1}(\lambda|1) \geq C_{n,k} \sigma_{k}(\lambda).
\end{equation*}

Thus, we have
\begin{equation*}
\label{maxw11}
\frac{\sigma_{k}^{11} b_{11}^{2}}{\sigma_{k}}
= \frac{\sigma_{k-1}(\lambda|1) \lambda_{1}^{2}}{\sigma_{k}}
\geq \frac{C_{n,k} \sigma_{k} \lambda_{1}}{\sigma_{k}}
= C_{n,k} b_{11}.
\end{equation*}
Hence,
\begin{equation}\label{Ti}
\sigma^{11}_{k} b^{2}_{11} \geq C^{*} b_{11},
\end{equation}
where the positive constant $C^{*}$ depends on $n$, $k$, $\underset{\Ctheta}{\min}h$, $\|h\|_{C^{1}(\mathcal{C}_\theta)}$, $\underset{\Ctheta}{\min}\hat{f}$, and $\|\hat{f}\|_{C^{2}(\mathcal{C}_\theta)}$.

Now inserting \eqref{Ti} into \eqref{Qi} and choosing
\[
A := \frac{\tilde{C}_{2}(W) + C_{19}K + 1}{C^{*}},
\]
we arrive at
\begin{equation}
\label{Ri}
\begin{aligned}
0 &\geq -\frac{\tilde{C}_{0}(W)}{\lambda_{1}}
      -\tilde{C}_{1}(W)
      -\tilde{C}
      -\tilde{C}_{2}(W)\lambda_{1}
      - C_{19}K \lambda_{1}
      + C^{*} A \lambda_{1} \\
&\geq -\frac{\tilde{C}_{0}(W)}{\lambda_{1}}
      -\tilde{C}_{1}(W)
      -\tilde{C}
      + \lambda_{1}.
\end{aligned}
\end{equation}
Therefore, whenever $\lambda_{1}$ is sufficiently large, \eqref{Ri} implies
$$
\lambda_{1} \leq C
$$
for some positive constant $C$.

{\bf Case II:} $\lambda_{k}\geq \delta_{k}\lambda_{1}$ with $\delta_{k}>0$. Since $\lambda_{1}\geq \lambda_{2}\geq \ldots \geq \lambda_{k}\geq \delta_{k}\lambda_{1}$,
$$
\hat{f}=\sigma_{k}>\lambda_{1}\ldots \lambda_{k}\geq \delta^{k-1}_{k}\lambda^{k}_{1},
$$
which also yields $\lambda_{1}\leq C$ for a positive constant $C$. Hence the proof of Theorem \ref{thm2} is complete.

\end{proof}

\begin{theorem}
Let $\theta \in (0,\frac{\pi}{2})$, let $h$ be a positive, even, capillary, strictly convex solution of Eq. \eqref{eq:capillary-Lp-dual-CM}. Then for any $\ga \in (0,1)$, there exists a positive constant $C$ depending only on $k, p,q, \underset{\Ctheta}{\min}f$ and $||f||_{C^{3}(\Ctheta)}$ such that
\begin{equation}\label{higher}
||h||_{C^{4,\ga}(\Ctheta)}\leq C.
\end{equation}
\end{theorem}
\begin{proof}
Together with Lemma \ref{lem3} and Theorem \ref{thm2}, we conclude that
$$||h||_{C^{2}(\Ctheta)}\leq C.$$
By applying the theory of fully nonlinear second-order uniformly elliptic equations with oblique derivative boundary condition \cite{LT}*{Theorem1.1}, we obtain a $C^{2,\ga}$($\ga \in (0,1)$) estimate for $h$, along with the higher-order estimates as in \eqref{higher}.
\end{proof}

\section{Existence of solutions}\label{sec 4}
In this section, we use a degree theory argument as in \cite{degree1,degree2} to complete the proof of Theorem \ref{thm1}, following the approach of \cite{g1,GuanMaZhou2006,g2}.
\begin{proof}[\textbf{Proof of Theorem  \ref{thm1}}]
For $t \in [0, 1]$, we consider a family of equations
 $$\sigma_{k}(\nabla^{2}h+hI)=f_{t}h^{p-1}(|\nabla h|^{2}+|h|^{2})^{\ga_{t}},$$
where 
$$f_{t}:=\left((1-t)\Big(\frac{\ell^{p-1}}{\binom{n}{k}}\Big)^{\frac{1}{p+k-1}}+tf^{-\frac{1}{p+k-1}}\right)^{-(p+k-1)},$$
and
$$\ga_{t}:=\frac{k+1-q_{t}}{2},\quad q_{t}:=k+1-t(k+1-q).$$
For $(h, t)$ with $h\in C^{l+2,\ga}_{\rm{even}}(\Ctheta)$, where $l\geq 0$ and $C^{l+2,\ga}_{\rm{even}}(\Ctheta)$ is the subset of even functions in $C^{l+2,\ga}(\Ctheta)$, we consider the following problem
\begin{equation*}\label{existence}
		\renewcommand{\arraystretch}{1.5}
		\left\{
		\begin{array}{rll}
		F(h,t)&=\sigma_{k}(\nabla^{2}h+hI)-f_{t}h^{p-1}(|\nabla h|^{2}+|h|^{2})^{\ga_{t}},\quad & \text{in}\quad\Ctheta,\\
		G(h,t)&=\nabla_{\mu}h-\cot\theta h,\quad & \text{on} \quad\pa\Ctheta. 
		\end{array}	
			\right.
	\end{equation*}
Now, let $\hat{R}>0$ be fixed, define $\mathcal{O}\subset C^{l+2,\ga}_{\rm{even}}(\Ctheta)$ as
$$\mathcal{O}=\{h \in C^{l+2,\ga}_{\rm{even}}(\Ctheta): ||h||_{C^{l+2,\ga}_{\rm{even}}(\Ctheta)}\leq \hat{R},\quad \nabla^{2}h+hI>0\}.$$	
From \eqref{higher} and Theorem \ref{full rank}, we know that if $\hat{R}$ is sufficiently large, 
$$\Big(F(h,t), G(h,t)\Big)\neq (0,0)\quad \text{for all}\quad (h,t) \in \pa\mathcal{O}\times [0,1].$$
Then using \cite{degree1}*{Theorem 1}, for each $t \in [0,1]$, there is a well-defined integer-valued degree as follows 
$$\text{deg}\Big(F(\cdot,0), G(\cdot,0), \mathcal{O}, 0\Big)=\text{deg}\Big(F(\cdot,t), G(\cdot,t), \mathcal{O}, t\Big).$$ 
By \cite{mei2025prescribedlpcurvatureproblem}*{Theorem 5.4}, we know that $h_{0}=\ell$ is the unique capillary even solution to $\Big(F(h,0), G(h,0)\Big)=(0,0)$. By \cite{mei2025prescribedlpcurvatureproblem}*{Lemma 5.6}, the linearized operator $\mathcal{L}:=D_{h}(F,G)(h_{0},0)$ has a trivial kernel and $\mathcal{L}$ is invertible. Thus from \cite{degree1}*{Theorem 1, Corollary 2.1}, we conclude that
$$\text{deg}\Big(F(\cdot,t), G(\cdot,t), \mathcal{O}, t\Big)\neq 0,\quad \forall t\in[0,1].$$
This implies that there exists $h\in \mathcal{O}$ such that
$$\Big(F(h,t), G(h,t)\Big)= (0,0) \quad \text{for all}\ t \in [0,1],$$
and in particular for $t=1$. 

Finally, we verify that $f_{t}$ satisfies conditions \eqref{23} and \eqref{24} to complete the proof of Theorem \ref{thm1}. Following a similar argument in \cite{mei2025prescribedlpcurvatureproblem}, we prove 
\begin{equation}\label{tt}
\nabla^{2}\ell^{\tilde{\lbd}}+\ell^{\tilde{\lbd}}\delta>0, \quad \forall \xi \in \Ctheta,
\end{equation}
where $\tilde{\lbd}:=\frac{p-1}{p+k-1}$. Fixing a point $\xi\in\Ctheta$, we can choose an orthonormal tangential frame $\{e_{i}\}_{i=1}^{n}$ around $\xi$, such that at $\xi$, $e_{i}\in T_{\xi}P$ ($1\leq i\leq n-1$) are the tangential vectors of level set $P:=\{\eta\in \Ctheta: \ell(\eta)=\ell(\xi)\}$, $e_{n}$ is the unit outward co-normal of $P\subset \Ctheta$. Then $\ell_{i}(\xi)=0$ for $1\leq i\leq n-1$, and $\ell_{n}=\cos\theta  \langle e_{n}, e\rangle:=\cos\theta \sin \widehat{\theta}$, which implies
\begin{equation}\label{l-lam}
\ell(\xi)=\sin^{2}\theta+\cos\theta \langle \xi , e\rangle=1-\cos\theta \cos\widehat{\theta}\geq \sin^{2}\widehat{\theta}.
\end{equation}
Combining \eqref{l-lam}, we obtain 
 \begin{align*}
     (\ell^{\tilde{\lambda}})_{ij}+\ell^{\tilde{\lambda}}\delta_{ij}&=\left(\tilde{\lambda}\ell^{\tilde{\lambda}-1}+(1-\tilde{\lambda})\ell^{\tilde{\lambda}}\right)\delta_{ij}+\tilde{\lambda}(\tilde{\lambda}-1)\ell^{\tilde{\lambda}-2}\ell_{i}\ell_{j}\\
     &\geq\tilde{\lambda} \ell^{\tilde{\lambda}-2} \left(\ell +(\tilde{\lambda}-1)\ell_{i}\ell_{j}   \right) \\ 
     &\geq \tilde{\lambda} \ell^{\tilde{\lambda}-2}  \left(\sin^{2}\widehat{\theta}+(\tilde{\lambda}-1)\cos^{2}\theta \sin^{2}\widehat{\theta} \right)\delta_{ij}>0,
 \end{align*}
 which shows that \eqref{tt} holds, and hence $f_{t}$ satisfies condition \eqref{23}. On the other hand, we can check that $f_{t}$ satisfies the boundary condition \eqref{24}. 

 \end{proof}
\section*{Conflict of Interest Statement}
The authors declare that they have no conflict of interest.

\section*{Data Availability Statement}
No datasets were generated or analyzed during the current study.
 
\section*{References}


\end{document}